\documentclass{article}
\usepackage[utf8]{inputenc}
\usepackage{amsmath}
\usepackage{amsthm}
\usepackage{amssymb}
\usepackage{tikz}
\usetikzlibrary{arrows, automata}
\usetikzlibrary{intersections}
\usepackage{hyperref}
\usepackage{cite}
\usepackage{authblk}
\usepackage{graphicx}
\usepackage{adjustbox}
\usepackage{caption}
\usepackage{subcaption}
\usepackage{appendix}
\usepackage{pgfplots}
\usepgfplotslibrary{fillbetween}
\pgfplotsset{compat=1.11}
\pgfdeclarelayer{bg}
\pgfsetlayers{bg,main}

\newtheorem{conjecture}{Conjecture}
\newtheorem{definition}{Definition}

\newtheorem{observation}{Observation}
\newtheorem{theorem}{Theorem}
\newtheorem{lemma}{Lemma}
\newtheorem{corollary}{Corollary}
\newtheorem{proposition}{Proposition}

\providecommand{\subjclass}[1]{\textbf{2010 AMS Subject Class:} #1}
\providecommand{\keywords}[1]{\textbf{Keywords:} #1}
\title{Dominator Chromatic Numbers of Orientations of Trees}
\author[1]{Michael Cary\footnote{macary@mix.wvu.edu}}
\affil[1]{West Virginia University}
\begin{document}
\maketitle
\begin{abstract}
In this paper we prove that the dominator chromatic number of every oriented tree is invariant under reversal of orientation. In addition to this marquee result, we also prove the exact dominator chromatic number for arborescences and anti-arborescences as well as bounds on other orientations of oft studied tree topologies including generalized stars and caterpillars.
\end{abstract}



\section{Introduction}
A dominating set of a graph $G=(V,E)$ is a subset $S\subseteq V$ such that every vertex is either adjacent to a vertex in $S$ or is in $S$ itself, i.e., $V\setminus S\subseteq N(S)$, or, alternatively, $S\cup N(S)=V$, where $N(S)$ is the open neighborhood of $S$. This notion can be extended to digraphs $D=(V,A)$ through out-neighborhoods by finding a subset $S\subseteq V$ such that every vertex is in either $N^{+}(S)$ or in $S$ itself, i.e., $S\cup N^{+}(S)=V$, where $N^{+}(S)$ is the open out-neighborhood of $S$.

Domination and dominating set problems come in a variety of flavors, and date back at least to a chess problem found in \cite{de1862traite} about which positions a queen can dominate on a chessboard. Since then, many applications of domination and dominating sets in graphs and networks have been discovered. Results on graph separability using dominating sets were obtained in \cite{chiarelli2019linear}. Progress on binary locating sets for vertices in polytopes occurred in \cite{simic2017binary}. In \cite{saygi2019domination} bounds were found for the dominator number of hypercube networks, in particular for Fibonacci cubes.

Characterizing these sets and their potential applications are interesting and important problems. For example, while every graph with no isolated vertex admits at least two disjoint dominating sets, this is not true in general for total dominating sets.
In fact, a recent paper by Henning and Peterin \cite{henning2019characterization} characterized graphs with two disjoint total dominating sets. But typical applications of dominating sets involve characterizing the vertices of dominating sets, not the dominating sets themselves. For this purpose, an interesting and relatively new concept was developed, that of dominator colorings.

A dominator coloring of a graph is a proper vertex coloring in which every vertex dominates at least one color class. The dominator chromatic number of a graph is the size of a minimum dominator coloring. Dominator colorings and dominator chromatic numbers of graphs were first introduced in papers by Dr. Gera \cite{gera2006,gera2007}. Since that initial work, dominator colorings have been studied intensively in specific graph families including bipartite graphs \cite{gera2007bi}, trees \cite{merouane2012}, and certain Cartesian products \cite{chen2017dominator}. More general results were obtained in \cite{arumugam2012dominator} and \cite{harary1996}. Algorithmic results were first obtained in \cite{arumugam2011algorithmic} and \cite{chang1998algorithmic} in a general setting as well as in \cite{merouane2015} in which a specific algorithm for finding minimum dominator colorings of trees was developed. These results led to the vast development of applications of dominating sets in undirected networks, including, e.g., \cite{blair2011movable,desormeaux2018distribution,haynes1998,haynes2002,haynes2003}. 
Recently the notion of dominator colorings was extended to directed graphs in \cite{cary2020dominator}. In that paper the focus was on finding the dominator chromatic number over all possible orientations of paths and cycles. In this paper, for the most part, we shift our attention from this perspective towards finding the exact dominator chromatic number of specific structures. We note that in \cite{cary2020linear} a linear algorithm for finding the minimum dominator coloring of an oriented path was developed and open source software implementing this algorithm was released. Before continuing, we provide a formal definition of a dominator coloring for a directed graph.

\begin{definition}\label{maindef}
A dominator coloring of a directed graph is a proper vertex coloring which further satisfies the requirement that every vertex dominates at least one color class in its out-neighborhood.
\end{definition}

Notice the importance of orientation here. If we did not specify that the dominator coloring was with respect to the out-neighborhood, then this problem would be analogous to that of dominator colorings of undirected graphs. This means that the definition of domination in a directed graph can be expressed in terms of the ordered pairs of vertices that comprise the arcs; for an arc $e=(u,v)$ we say that the vertex $u$ dominates the vertex $v$, but not vice versa.

All digraphs in this paper are orientations of simple, finite trees. One particular structure which will be a primary focal point in this paper is the arborescence. An arborescence, also known as an out-tree, is the orientation of a tree in which all arcs point away from a single source. Similarly, an anti-arborescence, or an in-tree, is the orientation of a tree in which all arcs point towards a single sink. As we will see, the dominator chromatic number of arborescences is invariant under reversal of orientation. 

This rather interesting results naturally begs to be generalized. Surprisingly yet beautifully, the marquee result in this paper is that the dominator chromatic number of an oriented tree $T$ is invariant under reversal for all oriented finite trees. 

Before continuing this work, we declare several important notations. Unless otherwise specified, the cardinality of the vertex set $V(T)$ of a given tree $T$ is denoted by $n$. The set of all leaves of $T$ is denoted by $l(T)$ and the number of leaves in a tree, i.e., $|l(T)|$, is denoted by $l$ when $T$ is known. For an oriented tree $T$, we will denote its reversal by $T^{-}$. The out-degree of a vertex $v$ is denoted by $d^{+}(v)$. Finally, we will use $\chi_{d}(D)$ to denote the dominator chromatic number of a given digraph $D$.

\section{Arborescences}
In this section we study arborescences, building up to a proof of the dominator chromatic number of any arborescence or anti-arborescence. First, however, we recap previous results on orientations of paths.

The simplest type of tree is a path. It is easy to see that the dominator chromatic number of the directed path $P_{n}$ is $n$ and that this result is trivially invariant under reversal. The following theorem from \cite{cary2020dominator} gives the dominator chromatic number over all orientations of paths.
\begin{theorem}\label{t1}
The minimum dominator chromatic number over all orientations of the path $P_{n}$ is given by
\begin{equation*}
\chi_{d}(P_{n})=\begin{cases}
k+2 & \mathrm{if}\ n=4k\\
k+2 & \mathrm{if}\ n=4k+1\\
k+3 & \mathrm{if}\ n=4k+2\\
k+3 & \mathrm{if}\ n=4k+3
\end{cases}
\end{equation*}
for $k\geq 1$ with the exception $\chi_{d}(P_{6})=3$.
\end{theorem}

Moving on from paths, especially when considering directed paths specifically, arborescences (and anti-arborescences) can be considered an immediate extension. As we build towards our proof that the dominator chromatic number of (anti-)arborescences is invariant under reversal of orientation, we first characterize the dominator chromatic number of (anti-)arborescences in terms of their parameters, namely the size of their vertex set and the number of leaves.

\begin{lemma}\label{l1}
Let $T$ be an in-tree (anti-arborescence). Then $\chi_{d}(T)=n-l+1$. 
\end{lemma}
\begin{proof}
Since the set of leaves of $T$ form an independent set and since $d^{-}(v)=0$ for all leaves $v$ ($T$ is an in-tree), we may assign the set $l(T)$ to the same color class. It remains to be shown that every non-leaf vertex of $T$ must be uniquely colored. This follows from the fact that $d^{+}(v)\leq 1$ for all $v\in V(T)$ when $T$ is an in-tree, as this implies that each non-leaf vertex is the only vertex dominated by some other vertex in $T$ and therefore must be uniquely colored.
\end{proof}
\begin{lemma}\label{l2}
Let $T$ be an out-tree (arborescence). Then $\chi_{d}(T)=n-l+1$.
\end{lemma}
\begin{proof}
The proof is by contradiction via minimum counterexample with respect to $n=|V(T)|$. Since a single vertex $v$ is technically an arborescence with one leaf and since $\chi_{d}(v)=1$, we may assume that some tree $T$ is a minimum counterexample to our claim. Let $|V(T)|=n$, let $v$ be a leaf of $T$ with in-neighbor $u$, and let $T^{\prime}=T\setminus\{v\}$. Since $|V(T^{\prime})|<|V(T)|$, it follows that $\chi_{d}(T^{\prime})=(n-1)-(l-1)+1=n-l+1$ if $v$ was not the only out-neighbor of $u$, and $\chi_{d}(T^{\prime})=(n-1)-l+1=n-l$ if $v$ was the only out-neighbor of $u$ in $T$. If $v$ was not the only out-neighbor of $u$, then we may color $v$ with the same color as the other out-neighbors of $u$, and if $v$ was the only out-neighbor of $u$ then we color $v$ uniquely. In either case we have established that $\chi_{d}(T)=n-l+1$ which contradicts our assumption that $T$ was a minimum counterexample and conclude that $\chi_{d}(T)=n-l+1$ for any arborescence $T$.
\end{proof}
\begin{corollary}\label{c1}
For any (anti-)arborescence $T$, we have that $\chi_{d}(T)=\chi_{d}(T^{-})$ where $T^{-}$ is $T$ with the orientation of every arc reversed.
\end{corollary}
Notice that Theorem \ref{t1} is a special case of both of the above lemmas and their corollary.

\section{Main Results}
In this section we prove our main result, that the dominator chromatic number of an oriented tree is invariant under reversal. To do this we will prove several necessary lemmas en route. First, however, we begin with an observation.

\begin{observation}
Let $T$ be an orientation of a tree and let $v$ be a leaf of $T$. Then $\chi_{d}(T)-1\leq\chi_{d}(T\setminus\{v\})\leq\chi_{d}(T)$.
\end{observation}
\begin{proof}
Clearly if the lower bound is false then we may find a smaller dominator chromatic number for $T$. The upper bound is obvious.
\end{proof}

Next, in order to prove our main theorem, we will need a more refined analysis of the subtree $T^{\prime}=T\setminus\{v\}$ for some leaf $v$ of $T$. We begin by first providing a characterization of the subtree $T^{\prime}$ when $\chi_{d}(T^{\prime})=\chi_{d}(T)-1$.

\begin{lemma}\label{lemT1}
Let $T$ be an orientation of a tree and let $v$ be a leaf of $T$ with neighbor $u$. Then $\chi_{d}(T\setminus\{v\})=\chi_{d}(T)-1$ if and only if either $\{v\}=N^{+}(u)$ or $\{v\}=\{x\in V(T)|d^{-}(x)=0\}$.
\end{lemma}
\begin{proof}
($\impliedby$) It is obvious that $v$ must be uniquely colored in any minimum dominator coloring of $T$ in either case.

($\implies$) Since $v$ is a leaf of $T$, $d^{+}(v)+d^{-}(v)=1$. If $d^{-}(v)=1$ it suffices to show that if $|N^{+}(u)|>1$ then $\chi_{d}(T\setminus\{v\})=\chi_{d}(T)$. This follows immediately by seeing that for any $x\in N^{+}(u)\setminus\{v\}$ it must be that if $v$ is uniquely colored, we may recolor $v$ with $c(x)$ thereby contradicting our assumption that our dominator coloring was a minimum dominator coloring. If $d^{+}(v)=1$ and there exists some other vertex $x$ such that $d^{-}(x)=0$, then if $v$ was uniquely colored, we may again color $v$ with $c(x)$ and contradict our assumption that our dominator coloring was a minimum dominator coloring of $T$.
\end{proof}

With that lemma intact, we further refine our understanding of the subtree $T^{\prime}=T\setminus\{v\}$ by studying the neighbor of the leaf vertex $v$.

\begin{lemma}\label{lemT2}
Let $T$ be an orientation of a tree with leaf $v$ satisfying $\chi_{d}(T\setminus\{v\})=\chi_{d}(T)-1$. If $d^{-}(v)=0$ then $d^{-}(u)=1$ where $u$ is the neighbor of $v$ in $T$.
\end{lemma}
\begin{proof}
Clearly $u$ cannot have any additional in-neighbors that are also leaves, else $v$ would not be uniquely colored in any minimum dominator coloring of $T$. Moreover, all leaves $l\neq v$ of $T$ must have $d^{-}(l)=1$, else $\{v\}\neq\{x\in V(T)|d^{-}(x)=0\}$ in which case Lemma \ref{lemT1} tells us that $\chi_{d}(T\setminus\{v\})\neq\chi_{d}(T)-1$, a contradiction.

Let $T^{\prime}$ be the subtree of $T$ obtained by deleting all leaves of $T$ except for $v$. If $T^{\prime}$ has a leaf $l$ with $d^{-}(l)=0$ then the vertex $l$ has $d^{-}(l)=0$ in $T$. By iterating this process until $T^{n}=vu$ we see that it must be the case that $d^{-}(u)=1$ in $T$.
\end{proof}

Finally we are ready to prove the main result of this paper, that the dominator chromatic number of an oriented tree is invariant under reversal.

\begin{theorem}\label{tmain}
Let $T$ be an orientation of a tree and let $T^{-}$ be $T$ with the orientation of every arc reversed. Then $\chi_{d}(T)=\chi_{d}(T^{-})$.
\end{theorem}
\begin{proof}
Let $T$ be a minimum counterexample, let $v$ be a leaf of $T$ with neighbor $u$, and let $T^{\prime}=T\setminus\{v\}$. We prove this by considering the cases where $\chi_{d}(T^{\prime})=\chi_{d}(T)-1$ and where $\chi_{d}(T^{\prime})=\chi_{d}(T)$. Within each case we will consider the two possible subcases, where $d^{-}(v)=0$ and where $d^{+}(v)=0$.

First, assume that $\chi_{d}(T^{\prime})=\chi_{d}(T)-1$. If $d^{-}(v)=0$ then by Lemma \ref{lemT1} we know that $\{v\}=\{x\in V(T)|d^{-}(x)=0\}$. This implies that $d^{-}(u)=1$ by Lemma \ref{lemT2}. Thus $d^{+}_{T^{-}}(u)=1$ and $v$ must be uniquely colored and so $\chi_{d}(T^{\prime})=\chi_{d}(T)$.

If $d^{+}(v)=0$ then $d^{+}(u)=1$ else $v$ would not be uniquely colored in $T$. Thus $d^{-}(u)=1$ in $T^{\prime-}$. If $u$ is not uniquely colored in $T^{\prime-}$, then we may uniquely recolor $u$ and assign to $v$ the original color had by $u$ in $T^{\prime-}$. Since $d^{-}(u)=0$ in $T^{\prime-}$, if there exists $x\in V(T^{\prime-})$ such that $c(x)=c(u)$, then this color class is not dominated by any vertex. Furthermore, if there exists any non-dominated color classes in $T^{\prime-}$ then $u$ must belong to such a class, else our dominator coloring of $T^{\prime-}$ is not minimal. Therefore, if $u$ is uniquely colored in a minimum dominator coloring of $T^{\prime-}$, there does not exist any non-dominated color class besides $c(u)$ in $T^{\prime-}$. This implies that $\{u\}=\{x\in V(T^{\prime-})|d^{-}(u)=0\}$ which implies that $v$ must be uniquely colored in order to have a proper dominator coloring of $T^{-}$, hence $\chi_{d}(T^{-})=\chi_{d}(T)$.

Now assume that $\chi_{d}(T^{\prime})=\chi_{d}(T)$. If $d^{+}(v)=0$ then $\exists\ x\in N^{+}(u)$ such that $x\neq v$ since $v$ is not uniquely colored in any minimum dominator coloring of $T$. Since $|V(T^{\prime})|<|V(T)|$ it follows that $\chi_{d}(T^{\prime-})=\chi_{d}(T^{\prime})=\chi_{d}(T)$. Since $xu\in A(T^{\prime}-)$ and $d^{+}(x)=1$ in $T^{\prime}-$, it follows that $u$ is uniquely colored in any minimum dominator coloring of $T^{\prime}-$. We may then add $v$ to $T^{\prime-}$ and color it with $c(x)$ to establish that $\chi_{d}(T^{-})=\chi_{d}(T)$.

If $d^{-}(v)=0$ then $\exists\ x\neq v$ such that $d^{-}(x)=0$ by Lemma \ref{lemT1}. If such a vertex is also an in-neighbor of $u$ in $T$ then we may color $v$ with $c(x)$ and be done, so we may assume that $\{v\}=N^{+}(u)$ in $T^{\prime-}$. Since $|V(T^{\prime})|<|V(T)|$, it follows that $\chi_{d}(T^{\prime-})=\chi_{d}(T^{\prime})$. Since $\{v\}=N^{-}(u)$ in $T$, it follows that $d^{-}(u)=0$ in $T^{-}$. This mean that $|\{w\in V(T^{\prime})|d^{-}(w)=0\}|\geq 2$. If $x$ is a leaf in $T$, then $x$ cannot be the only in-neighbor of its out-neighbor in $T^{\prime}$, else $x$ would be uniquely colored in any minimum dominator coloring of $T^{\prime}$ and we could consider the tree $T\setminus\{x\}$ which satisfies $\chi_{d}(T\setminus\{x\})<\chi_{d}(T)$ and use the first case to complete the proof. Therefore we may assume that $\chi_{d}(T^{\prime\prime-})=\chi_{d}(T^{\prime\prime})$ where $T^{\prime\prime}=T\setminus\{x\}$. Let $w$ be the in-neighbor of $x$ in $T^{-}$. Since $x$ is not the only member of $N^{+}(w)$, $w$ already dominates a color class in any minimum dominator coloring of $T^{\prime\prime-}$, all one such color class $\hat{c}$. By coloring $x$ with $\hat{c}$ in a given minimum dominator coloring of $T^{-}$, we establish that if $x$ is a leaf in $T$, that $\chi_{d}(T^{-})=\chi_{d}(T)$.

Finally, assume that there is no such leaf vertex in $T$, i.e., that for all leaves $l\neq v$ of $T$, $d^{-}(l)=1$. If there exists some leaf $l\neq v$ such that $\{l\}=N^{+}(N^{-}(l))$, the second subcase of the first case completes the proof. Therefore, we may assume that for every leaf $l\neq v$ (which there necessarily exists at least one such leaf) we have that there exists some other leaf $l^{\prime}$ such that $l^{\prime}$ is also a member of $N^{+}(N^{-}(l))$. Let $T^{\prime\prime}=T\setminus\{l\}$. Since $|V(T^{\prime\prime})|<|V(T)|$, it follows that $\chi_{d}(T^{\prime\prime-})=\chi_{d}(T^{\prime\prime})$. We may color $l$ with $c(l^{\prime})$ for some $l^{\prime}\in N^{-}(N^{
+}(l))$ in $T^{-}$, establishing that $\chi_{d}(T^{-})=\chi_{d}(T)$. As this concludes the last possible case, the proof is complete.
\end{proof}

\section{Applications}
In this section we use our main result, that the dominator chromatic number of an orientation of a tree is invariant under reversal of orientation, as a tool to help prove results on other oft studied tree topologies including generalized stars and caterpillars.

We begin by studying orientations of generalized stars. In some sense a generalized star could be construed to mean a union of paths of varying lengths that all meet at a single common vertex. However, for the sake of this paper we consider generalized stars $GS_{m}^{k}$ which consist of $m$ paths consisting of $k$ edges, all originating from a single, common vertex. Note that the traditional star graph $S_{m}$ is also $GS_{m}^{1}$. An illustration of $GS_{8}^{2}$ is given below for reference.

\begin{figure}[h!]
\centering
\begin{tikzpicture}[-,>=stealth',shorten >=1pt,auto,node distance=2cm,
                    thick,main node/.style={circle,draw}]
  \node[main node] (A)					      {};
  \node[main node] (B) [above of=A]		      {};
  \node[main node] (C) [above of=B]      	  {};
  \node[main node] (D) [above right of=A] 	  {};
  \node[main node] (E) [above right of=D]     {};
  \node[main node] (F) [above left of=A]      {};  
  \node[main node] (G) [above left of=F]      {};    
  \node[main node] (H) [below left of=A]      {};    
  \node[main node] (I) [below left of=H]      {};  
  \node[main node] (J) [below of=A]           {};    
  \node[main node] (K) [below of=J]           {};    
  \node[main node] (L) [below right of=A]     {};    
  \node[main node] (M) [below right of=L]     {};
  \node[main node] (N) [left of=A]            {};
  \node[main node] (O) [left of=N]            {};
  \node[main node] (P) [right of=A]           {};
  \node[main node] (Q) [right of=P]           {};
  \draw[thick,-] (A) to (B);
  \draw[thick,-] (B) to (C);
  \draw[thick,-] (A) to (D);
  \draw[thick,-] (D) to (E);
  \draw[thick,-] (A) to (F);
  \draw[thick,-] (F) to (G);
  \draw[thick,-] (A) to (H);
  \draw[thick,-] (H) to (I);
  \draw[thick,-] (A) to (J);
  \draw[thick,-] (J) to (K);
  \draw[thick,-] (A) to (L);
  \draw[thick,-] (L) to (M);
  \draw[thick,-] (A) to (N);
  \draw[thick,-] (N) to (O);
  \draw[thick,-] (A) to (P);
  \draw[thick,-] (P) to (Q);
\end{tikzpicture}
\caption{An example of an undirected generalized star $GS_{8}^{2}$. The central vertex is an end vertex of all eight paths.}
\label{f1}
\end{figure}
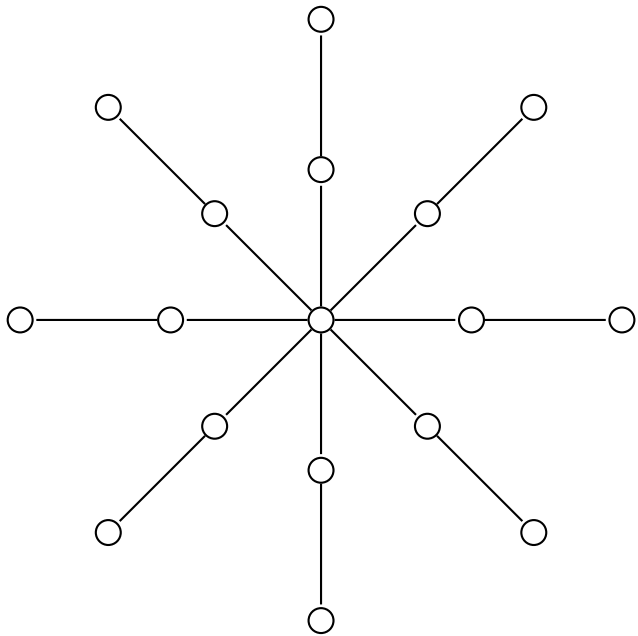

The following result from \cite{cary2020dominator} tells us the exact dominator chromatic number of a star for every possible orientation.

\begin{proposition}
Let $D$ be an orientation of a star graph, $G$. Then we have that $2\leq\chi_{d}(D)\leq3$, $\chi_{d}(D)=2$ if and only if all arcs are oriented similarly with respect to the central vertex, and $\chi_{d}(D)=3$ otherwise.
\end{proposition}
It is easy to see from this result that the dominator chromatic number of an oriented star is invariant under reversal. We can now attempt to expand this result by proving further results on the dominator chromatic number of generalized stars $GS_{m}^{k}$.

\begin{lemma}
Let $GS_{m}^{k}$ be an orientation of generalized star featuring either a single source or a single sink. Then $\chi_{d}(GS_{m}^{k})=m(k-1)+2$.
\end{lemma}
\begin{proof}
This follows from Lemmas \ref{l1} and \ref{l2} and from Theorem \ref{tmain}. Since $GS_{m}^{k}$ has $m$ paths extending from its central vertex, it has $m$ leaves. Since $|V(GS_{m}^{k})|=n=mk+1$, it follows that $\chi_{d}(GS_{m}^{k})=mk-m+2=m(k-1)+2$.
\end{proof}

A natural question to ask would be whether or not this is best possible. It turns out that this is not the case. We improve greatly on this bound with our next result.

\begin{lemma}
For $k\geq 2$ we have that the minimum dominator chromatic number of all possible orientations of the generalized star satisfies $\chi_{d}(GS_{m}^{k})\leq 3+m(\lfloor\frac{k}{2}\rfloor-1)$.
\end{lemma}
\begin{proof}
The proof is by construction. Begin by considering $GS_{m}^{2}$. Let $S_{0}$ be the central vertex of $GS_{m}^{2}$, let $S_{1}=N(S_{0})$, and let $S_{2}=N(S_{1})\setminus S_{0}$. Notice that we have simply created independent sets of ``layers" of the generalized star $GS_{m}^{2}$. We orient all arcs away from $S_{1}$ into $S_{0}$ and $S_{2}$. We may assign a single color to $S_{1}$ since no vertex has positive in-degree, and we may assign a single color to $S_{0}$ as it is a single vertex. Since every vertex in $S_{1}$ dominates, $S_{0}$, we may color all of $S_{2}$ with a third color, demonstrating that $\chi_{d}(GS_{m}^{2})\leq 3$.

We generalize this structure by coloring all vertices in each $S_{2i+1}$ with the same color used on $S_{1}$,  by coloring all vertices in each $S_{2i}$ with a unique color, and by orienting all arcs from $S_{2i+1}$ to $S_{2i}$. Notice that the union $\bigcup S_{2i+1}$ constitutes an independent set comprised of vertices with in-degree equal to zero, hence they may all share the color assigned to $S_{1}$. Since all vertices in $\bigcup S_{2i}$ have out-degree zero and are uniquely colored, every vertex in $\bigcup S_{2i+1}$ dominates some uniquely colored vertex. The result follows from basic counting.
\end{proof}

We next turn our attention to orientations of caterpillars. A caterpillar is a tree in which every vertex is at distance at most one from a central path. Our first result is a lemma relating the dominator chromatic number of this central path to the oriented caterpillar itself.

\begin{lemma}\label{lemlower}
Let $T$ be an orientation of a caterpillar and let $P\subseteq T$ be a longest path in $T$. Then $\chi_{d}(T)\geq\chi_{d}(P)$.
\end{lemma}
\begin{proof}
In order for $\chi_{d}(T)$ to be less than $\chi_{d}(P)$ the addition of a vertex to $P$ must allow for the combination of existing color classes in $P$, but if this is possible then it must be possible without the addition of a new vertex which contradicts $P$ using fewest possible colors.
\end{proof}

Next we bound the dominator chromatic number of an oriented caterpillar by a function of the length of its central (longest) path.

\begin{lemma}\label{lemupper}
Let $T$ be an orientation of a caterpillar and let $P\subseteq T$ be its central path with length $m$. Then $\chi_{d}(T)\leq 2m-1$.
\end{lemma}
\begin{proof}
Clearly $P$ requires no more than $m$ colors. For all vertices $v\in V(T)\setminus V(P)$ we have that either $d^{-}(v)=0$ in which case all such vertices $v$ may be assigned a new color, $c_{m+1}$. Since any other vertices remaining in $V(T)\setminus V(P)$ have $d^{-}(v)=1$, it follows that for each $u\in V(P)$ we may color all members of $N^{+}(u)\cap [V(T)\setminus V(P)]$ with the same color class. Since the end vertices of $P$ do not have any such neighbors (else they are not end vertices), we may color $V(T)\setminus V(P)$ with at most $m-2$ additional colors. Collectively, this worst case coloring requires $2m-1$ colors, thus establishing the upper bound on the dominator chromatic number of oriented caterpillars.
\end{proof}

An interesting case of this result arises when the central path of the caterpillar is a directed path. In this case we get the rather nice result that the lower bound from Lemma \ref{lemlower} is sharp.

\begin{lemma}
If $T$ is an orientation of a caterpillar whose central path $P$ is a directed path, then $\chi_{d}(T)=\chi_{d}(P)$.
\end{lemma}
\begin{proof}
Let $P=v_{1}\dots v_{m}$. Since $P$ is a directed path, every vertex in $P$, except for $v_{m}$, dominates not only a color class, but a single vertex. Notice also that the vertex $v_{1}$ is not dominated by any other vertex, else $P$ is not maximum (recall that $P$ does not need to be a directed path in order to be the central path of a caterpillar). Thus every vertex in $V(T)\setminus V(P)$ that dominates a vertex in $P$ also dominates a color class, and may be colored with $c(v_{1})$. Since every vertex in $T$ that has positive out-degree now dominates a color class, we may color all remaining vertices with $\hat{c}$, establishing that $\chi_{d}(T)=\chi_{d}(P)=m$.
\end{proof}

\section{Conclusion}
This paper initiated the study of dominator coloring of orientations of trees. Initially, a study of arborescences and anti-arborescences found that, for a given vertex set, the dominator chromatic number of an arborescence is equal to the dominator chromatic number of an anti-arborescence. We then generalized this finding, proving the most important result in this paper, that the dominator chromatic number of orientations of trees is invariant under reversal of orientation. Using this result, several results on the dominator chromatic number of generalized stars and caterpillars were obtained.

In our study of generalized stars $GS_{m}^{k}$, we established an upper bound on the dominator chromatic number over all orientations of $GS_{m}^{k}$. We conjecture that this result is the best possible.

\begin{conjecture}
The minimum dominator chromatic number over all orientations of the generalized star $GS_{m}^{k}$ is given by $\chi_{d}(GS_{m}^{k})=3+m(\lfloor\frac{k}{2}\rfloor-1)$.
\end{conjecture}

The study of dominator chromatic numbers of other structures which are derived from stars and generalized stars, such as wheels, are a natural extension of the results obtained on generalized stars in this paper. Additionally, the orientations of stars which are (anti-)arborescences are interesting in that they have relatively large dominator chromatic numbers. We conjecture that these two particular orientations are actually the worst possible for generalized star.

\begin{conjecture}
The maximum dominator chromatic number over all orientations of the generalized star $GS_{m}^{k}$ is given by $m(k-1)+2$ and occurs when $GS_{m}^{k}$ is an (anti-)arborescence.
\end{conjecture}

This conjecture could prove quite insightful in determining the nature of dominating sets in (anti-)arborescences. It seems to the author that these classes of orientations of trees might be among the worst possible when it comes to finding relatively small dominating sets.

With respect to caterpillars, we established several bound on the dominator chromatic number with respect to the central path of the caterpillar. Possible directions to extend the results in this paper could include finding an explicit minimum dominator chromatic number over all orientations of caterpillars and finding explicit dominator chromatic numbers for specific caterpillars with respect to their central paths. Additionally, the study of dominator chromatic numbers of lobsters could provide a useful stepping stone from paths and caterpillars to the study of dominator chromatic numbers of orientations of trees in general, and possibly even more broadly to orientations of directed acyclic graphs.

\bibliography{DCD2}
\end{document}